\documentclass[a4paper,10pt]{article}
\usepackage[utf8]{inputenc}
\usepackage{pdflscape}
\usepackage{comment,xcomment,amssymb,amsmath,color}
\usepackage{rotating,xypic,array}
\usepackage{latexsym,mathtools,amsthm}
\usepackage{subfig,rotating}
\usepackage{tikz,fancyvrb}
\usepackage{paralist}
\usepackage{booktabs}  
\usetikzlibrary{arrows.meta}
\usepackage{enumitem}
\usepackage{url} 
\usepackage{longtable}
\usepackage[section]{placeins}
\usepackage[longtable]{multirow}
\usepackage[colorlinks,citecolor=blue]{hyperref}

\title{Einstein metrics and Killing spinors on pseudo-Riemannian solvmanifolds}

\author{Diego Conti, Federico A. Rossi and Romeo Segnan Dalmasso}

\newtheorem{theorem}{Theorem}[section]
\newtheorem*{theorem*}{Theorem}
\newtheorem{lemma}[theorem]{Lemma}

\newtheorem{corollary}[theorem]{Corollary}
\newtheorem*{corollary*}{Corollary}
\newtheorem{proposition}[theorem]{Proposition}
\theoremstyle{definition}

\theoremstyle{remark}
\newtheorem{remark}[theorem]{Remark}

\newcommand{\abs}[1]{\left\vert#1\right\vert}
\newcommand{\R}{\mathbb{R}}
\newcommand{\im}{\mathrm{Im}\,}         
\newcommand{\lie}[1]{\mathfrak{#1}}     
\newcommand{\g}{\lie{g}}
\newcommand{\Lie}{\mathcal{L}}          

\newcommand{\N}{\mathbb{N}}

\newcommand{\C}{\mathbb{C}}

\newcommand{\LieG}[1]{\mathrm{#1}}      

\newcommand{\Spin}{\mathrm{Spin}}
\newcommand{\SU}{\mathrm{SU}}
\newcommand{\Sp}{\mathrm{Sp}}

\newcommand{\SO}{\mathrm{SO}}

\newcommand{\Cl}{\mathrm{Cl}}
\newcommand{\CCl}{\mathrm{\C l}}
\newcommand{\Gtwo}{\mathrm{G}_2}

\newcommand{\id}{\operatorname{Id}}   

\newcommand{\Span}[1]{\operatorname{Span}\left\{#1\right\}}

\newcommand{\D}[1]{\frac{\partial}{\partial#1}}
\DeclareMathOperator{\ric}{ric} 
\DeclareMathOperator{\Ric}{Ric} 

\DeclareMathOperator{\Hom}{Hom}
\DeclareMathOperator{\diag}{diag}

\DeclareMathOperator{\ad}{ad}

\DeclareMathOperator{\Tr}{tr}

\DeclareMathOperator{\grad}{grad}

\newcolumntype{C}{>{$}c<{$}}
\newcolumntype{L}{>{$}l<{$}}
\newcolumntype{R}{>{$}r<{$}}

\usepackage[backend=biber, maxbibnames=9, maxcitenames=9]{biblatex}
\begin{document}
\VerbatimFootnotes
\maketitle

\begin{abstract}
	Riemannian Einstein solvmanifolds can be described in terms of nilsolitons, namely nilpotent Lie groups endowed with a left-invariant Ricci soliton metric. This characterization does not extend to indefinite metrics; nonetheless, nilsolitons can be defined and used to construct Einstein solvmanifolds of a higher dimension in any signature.

	An  Einstein solvmanifold obtained by this construction turns out to satisfy the pseudo-Iwasawa condition, meaning that its Lie algebra splits as the orthogonal sum of a nilpotent ideal and an abelian subalgebra, the latter acting by symmetric derivations.
	
	In this paper we construct a family of pseudo-Iwasawa solvmanifolds admitting a Killing spinor in any dimension and signature and prove that all pseudo-Iwasawa solvmanifolds admitting a Killing spinor, invariant or not, belong to this family. If in addition the metric is Einstein, we show that the only possibility is the hyperbolic  half-space. As a byproduct, we prove that the only homogeneous Riemannian manifold admitting a Killing spinor with imaginary Killing constant is hyperbolic space.
\end{abstract}

\renewcommand{\thefootnote}{\fnsymbol{footnote}}
\footnotetext{\emph{MSC class 2020}: \emph{Primary} 53C25; \emph{Secondary} 53C50, 53C27, 22E25, 53C30.
}
\footnotetext{\emph{Keywords}: Killing spinor, Einstein metric,  pseudo-Riemannian metric, solvmanifold}
\renewcommand{\thefootnote}{\arabic{footnote}}

\section*{Introduction}

A Killing spinor on an $n$-dimensional pseudo-Riemannian spin manifold $(M,g)$ is a non-zero section $\psi$ of the spinor bundle $\Sigma M$ such that for some $\lambda\in\C$
\begin{equation}\label{eqn:Killingeqn}
	\nabla_X\psi=\lambda X\cdot\psi
\end{equation}
for any $X\in\Gamma(TM)$, where $\nabla$ is the connection on the spinor bundle induced by the Levi-Civita connection of $g$ and the dot denotes Clifford multiplication. As a matter of convention, we will exclude the degenerate case in which $\lambda=0$ and the spinor is parallel. We consider pseudo-Riemannian metrics of signature $(p,q)$; our study is mostly aimed at indefinite metrics, namely those for which $pq\neq0$, but also includes the Riemannian (or definite) case.

Killing spinors are of interest both in physics and mathematics. In physics they have been studied in relation to general relativity since~\cite{Walker1970OnSpacetimes} and later on in supergravity (see~\cite{Duff1986Kaluza-KleinSupergravity}). On compact Riemannian manifolds, Killing spinors correspond to eigenvectors for the Dirac operator which realize the lowest possible eigenvalue (see~\cite{friedrich1980}).

The existence of a Killing spinor $\psi$ puts strong constraints on the curvature; indeed, the Ricci operator satisfies (see \cite[Theorem~8]{BaumFriedrichGrunewaldKath})
\begin{equation}\label{eqn:RicciPsi}
	\Ric(X)\cdot\psi=4(n-1)\lambda^2X\cdot\psi.
\end{equation}
This forces  $\Ric-4(n-1)\lambda^2\id$ to be  a section of $\Hom(TM,V_\psi)$, with $V_\psi$ denoting the  distribution that annihilates $\psi$; the Clifford identity implies that elements of $V_\psi$ are isotropic, hence zero in definite signature. Thus, for Riemannian metrics~\eqref{eqn:RicciPsi} implies that $\Ric$ is a multiple of the identity, i.e. the metric is Einstein. An example of a non-Einstein Lorentz manifold endowed with a Killing spinor appears in~\cite{Bohle2003KillingManifolds}. Regardless of the signature, equation~\eqref{eqn:RicciPsi} implies that the scalar curvature satisfies
\begin{equation}
	\label{eqn:scalarifKilling}
	s=4n(n-1)\lambda^2.
\end{equation}
In particular, $\lambda$ is real or purely imaginary according to whether $s$ is positive or negative; by the same token, parallel spinors give $s=0$.

Killing spinors are also tied to the special holonomy groups $\SU(n)$, $\Sp(n)$, $\Gtwo$ and $\Spin (7)$. In the Riemannian case, a complete manifold $(M,g)$ with a Killing spinor for which $\lambda$ is imaginary is either hyperbolic space or a warped product of the interval with a special holonomy manifold of dimension $n-1$ (\cite{Baum_imaginary}). If $\lambda$ is real, the cone over $(M,g)$ is an $n+1$-dimensional special holonomy manifold, and the metric $g$ is either a round sphere, Einstein-Sasaki, nearly-K\"ahler, or nearly parallel $\Gtwo$ (\cite{Bar1993RealHolonomy}). Underlying these results is the characterization of special holonomy metrics in terms of parallel spinors established in~\cite{wangParallelSpinors}.
For Lorentzian signature, the list of holonomy groups of metrics admitting parallel spinors has been obtained in~\cite{Leistner:Classification}; a description of the Lorentzian geometries determined by a Killing spinor, entailing both analogies and complications when compared to the Riemannian case, appears in~\cite{baumLorentzian2000,baum2008codazzi,Bohle2003KillingManifolds}. In general signature, a full characterization seems more difficult to obtain due to the existence of indecomposable, not irreducible holonomy groups; however,  the irreducible holonomy groups corresponding to parallel spinors are indeed ``special'', in the sense that they have the same complexification as $\SU(n)$, $\Sp(n)$, $\Gtwo$ and $\Spin (7)$ (\cite{Baum_Kath_1999}).

This paper is focused on homogeneous metrics: we assume that the group of isometries acts transitively and a Killing spinor exists, not necessarily invariant under the group.
The Riemannian case is largely understood. As mentioned above, there are four possible geometries associated to a Killing spinor with $\lambda$ real (\cite{Bar1993RealHolonomy}); each of them is classified in the homogeneous case, leading to the following possibilities: round spheres, Einstein-Sasaki manifolds arising as
$\LieG{U}(1)$ bundles over generalized flag manifolds (\cite{Boyer_Galicki_2000}), nearly-K\"ahler symmetric spaces (\cite{Butruille_2005}), and the homogeneous nearly parallel $\Gtwo$ manifolds classified in~\cite{Friedrich_Kath_Moroianu_Semmelmann_1997}.
For $\lambda$ imaginary, a similar classification appears to be missing: taking~\cite{Baum_imaginary} into account, since homogeneous Riemannian manifolds are complete, the problem reduces to classifying homogeneous warped products admitting a Killing spinor with $\lambda$ imaginary. Such  manifolds are in particular Einstein of negative scalar curvature, as well as homogeneous, and therefore solvmanifolds by~\cite{Bohm_Lafuente_2023},  namely solvable Lie groups with a left-invariant metric. In general, the study of Riemannian Einstein solvmanifolds reduces to the study of  nilpotent Lie groups endowed with a Ricci soliton invariant metric, or nilsolitons; indeed, it follows from~\cite{Lauret_2010} and~\cite{Heber_1998} that every Einstein solvmanifold is, up to isometry, an orthogonal semidirect product of a nilsoliton and an abelian factor. There are many Riemannian Einstein metrics that can be obtained in this way (see \cite{Will_2003,FernandezCulma_2014,Conti_Rossi_2022,Arroyo:Filiform,KadiogluPayne:Computational} for partial classification results in low dimensions). This leaves the question of classifying those that admit a Killing spinor.

A first result in this paper is that only hyperbolic space occurs: the only homogeneous Riemannian  manifold that admits a Killing spinor with $\lambda$ imaginary is hyperbolic space. We prove this in two distinct ways. The first proof (Section~\ref{sec:homogeneous}) uses standard tools of Riemannian geometry and~\cite{Baum_imaginary}; the second proof (Section~\ref{sec:killingNotIwasawa}) is based on \cite{Bohm_Lafuente_2023,Lauret_2010,Heber_1998} and a more general study of pseudo-Riemannian Einstein solvmanifolds, which is the core of our paper.

Compared to the Riemannian situation, the pseudo-Riemannian setting is both more complicated and richer. Homogeneous indefinite metrics need not be complete (see \cite[§11]{Wolf_1964}).
In particular, the solvmanifold associated to hyperbolic space is the half-space model (see the first part of Section~\ref{sec:construction}), which is not complete: hyperbolic space of indefinite signature $(p,q)$ is diffeomorphic to $\R^{p}\times S^q$ (see \cite[Lemma~2.4.6]{Wolf:SpacesOfConstantCurvature}), and the half-space model can be identified with a proper open subset, as illustrated in~\cite{SeppiTrebeschi}. Beside hyperbolic half-space, Einstein solvmanifolds admitting a Killing spinor include other metrics with nonconstant curvature (see~\cite{Conti_Rossi_Segnan_2023}).
Moreover, indefinite Einstein solvmanifolds do not exhibit the rigid structure of their definite counterparts.  We will say that a  solvmanifold is pseudo-Iwasawa if it is the orthogonal semidirect  product of a nilpotent ideal and an abelian factor acting by symmetric derivations; if a pseudo-Iwasawa solvmanifold is Einstein, the nilpotent ideal is a nilsoliton (\cite{ContiRossi_IndeNilsEinsSov}). In contrast to the Riemannian case, not all indefinite homogeneous Einstein manifolds of negative scalar curvature fit into the pseudo-Iwasawa regime, even if one restricts to solvmanifolds (see \cite{ContiRossi_IndeNilsEinsSov,Rossi_2025,Conti_Rossi_SegnanDalmasso_2024}).
Thus, nilsolitons only form a part of the picture, although an important one constructively, since the Riemannian theory can be adapted to produce plenty of pseudo-Iwasawa Einstein solvmanifolds (\cite{Conti_Rossi_2022}).

It is natural to ask how these two layers of flexibility compare, namely whether it is possible to extend an indefinite nilsoliton to construct an Einstein solvmanifold that admits a Killing spinor, beyond the hyperbolic half-space. A partial indication to the contrary was given in~\cite{Conti_Rossi_Segnan_korean}, where we proved that Einstein-Sasaki solvmanifolds are not pseudo-Iwasawa. 

In this paper we generalize this result and classify all pseudo-Iwasawa solvmanifolds admitting Killing spinors. In Section~\ref{sec:construction} we construct a family of connected, simply connected Lie groups $\tilde G^{a,k}$ whose Lie algebra takes the form $\R^n\rtimes_{a\id+f}\R$ where $a\in\R$ and $f$ is a symmetric, nilpotent endomorphism of $\R^n$ with rank $k\geq0$. We show that the Lie group $\tilde G^{a,k}$ endowed with the left-invariant metric
\begin{equation*}\label{eqn:metricaGeneralHyper}
	e^{-2at}\biggl(\sum_{i=1}^k \delta_i(dx^i-tdz^i)\odot dz^i + \sum_{j=1}^{n-2k} \epsilon_j dy^j\otimes dy^j\biggr)+\epsilon_0dt^2,
\end{equation*}
admits one or more Killing spinors when $a\ne0$; we determine the precise number in Proposition~\ref{prop:counterexamplestomaintheorem}. When $a=0$ we obtain parallel spinors; in the particular case of Lorentzian signature, these metrics belong to the class of Brinkmann waves (see Remark~\ref{rmk:brinkmann}). In Section~\ref{sec:killingNotIwasawa} we prove that all pseudo-Iwasawa solvmanifolds admitting a Killing spinor belong to this family, and are Einstein only for $k=0$, corresponding to hyperbolic half-space (Theorem~\ref{thm:main}).

Applying this last result to the Riemannian case, we obtain our second proof of the fact that the only homogeneous Riemannian manifold carrying a Killing spinor with $\lambda$ imaginary is hyperbolic space.

In the indefinite case, the theorem shows that ad hoc constructions circumventing the pseudo-Iwasawa condition such as those of~\cite{Conti_Rossi_Segnan_2023} are indeed necessary in order to obtain new Einstein solvmanifolds admitting a Killing spinor.

\subsection*{Acknowledgments}
The authors are partially supported by the PRIN project n. 2022MWPMAB ``Interactions between Geometric Structures and Function Theories''. The authors also acknowledge Gruppo Nazionale per le Strutture Algebriche, Geometriche e le loro Applicazioni (GNSAGA) of Istituto Nazionale di Alta Matematica (INdAM).

\section{Homogeneous Riemannian metrics admitting Killing spinors}
\label{sec:homogeneous}
In this section we recall the construction of the bundle of spinors on a pseudo-Riemannian manifold, following~\cite{Lawson_Michelsohn_1989}, and some properties that will be needed in the proof of our theorem.  Then, we show that the only homogeneous Riemannian manifold carrying a Killing spinor with $\lambda$ imaginary is hyperbolic space.

Let $e_1,\dotsc, e_n$ be the standard basis of $\R^n$, let $e^1,\dotsc, e^n$ be the dual basis, and fix a scalar product
\begin{equation*}
	\label{eqn:generaldiagonalmetric}
	g=\epsilon_1e^1\otimes e^1+\dotsb  + \epsilon_ne^n\otimes e^n,
\end{equation*}
where $\epsilon_i=\pm1$. The metric induces musical isomorphisms $\flat\colon \R^n\to(\R^n)^*$, $\sharp\colon(\R^n)^*\to \R^n$, one the inverse of the other, defined by
\[g(v,w)=v^\flat(w), \qquad g(\alpha^\sharp, w)=\alpha(w).\]
If the metric has signature $(p,q)$, we will write $\R^{p,q}$ instead of $\R^n$ and denote by $\SO(p,q)$ the group of orientation preserving isomorphisms of $\R^n$ that preserve $g$. For any linear map $f\colon\R^{p,q}\to\R^{p,q}$, we will denote by $f^*$ its metric transpose and by $f^s$ its symmetric part, namely
\[f^s=\frac12(f+f^*).\]

Let $\Cl(p,q)$ denote the real Clifford algebra of $\R^{p,q}$, with the sign convention that $v\cdot w+w\cdot v=-2g(v,w)$. By definition, the spinor representation $\Sigma$ is an irreducible complex representation of $\Cl(p,q)$ (see e.g. \cite[§II.4]{Lawson_Michelsohn_1989}). We can identify $\R^{p,q}$ with a vector subspace of $\Cl(p,q)$, so that its action on $\Sigma$ restricts to Clifford multiplication, i.e. a bilinear map
\begin{equation*}
	\label{eqn:clifford}
	\R^{p,q}\otimes\Sigma\to\Sigma, \qquad (v,\psi)\mapsto v\cdot\psi.
\end{equation*}
The sign of this map depends on the orientation; we will always work in a fixed orthonormal basis, which will be assumed to be positively oriented. The group $\Spin(p,q)$ is the subgroup of $\Cl(p,q)$ generated by products of two unit vectors in $\R^{p,q}$; we denote by $\xi\colon \Spin(p,q)\to \SO(p,q)$ the double covering map. The Lie algebra of $\Spin(p,q)$ is generated by elements $e_i\cdot e_j$, acting on $\R^{p,q}$ as
\begin{equation}
	\label{eqn:xistare}
	\xi_{*e}(e_i\cdot e_j)v = 2\bigl(g(e_i,v)e_j-g(e_j,v)e_i\bigr)
\end{equation}
(see \cite[Proposition~I.6.2]{Lawson_Michelsohn_1989}).

Let $(M,g)$ be a pseudo-Riemannian oriented manifold of signature $(p,q)$. Let $P_{\SO}$ be the bundle of oriented orthonormal frames. A spin structure on $TM$  is a principal $\Spin(p,q)$-bundle $P_\Spin\to M$ with a $2:1$ bundle map $\tilde\xi$,
\[\xymatrix{  P_\Spin\ar[dr]\ar[rr]^{\tilde\xi} && P_\SO\ar[dl] \\ & M }\]
such that  $\tilde\xi(uh)=\tilde\xi(u)\xi(h)$ for all $h$ in $\Spin(p,q)$; the same construction can be applied to more general vector bundles than $TM$ (see \cite[§II.1]{Lawson_Michelsohn_1989}). A spinor is a section of the associated bundle $\Sigma M=P_\Spin\times_{\Spin(p,q)}\Sigma$.

The Levi-Civita connection on $P_\SO$ induces a connection on $P_\Spin$. Given a local section of $P_\Spin$, composing with $\tilde\xi$ we obtain a local section of $P_\SO$. This gives local trivializations $\{u_h\}$ and $\{e_j\}$ of the spinor and tangent bundles, implying that any spinor locally takes the form $\sum_h a_h u_h$, where each $a_h$ is a $\C$-valued smooth function. The Levi-Civita connection form $\theta$ is defined by $\nabla e_j=\sum_k\theta_{kj}\otimes e_k$, and as a consequence of~\eqref{eqn:xistare} the connection on a spinor $\psi=\sum_h a_hu_h$ is given by
\begin{equation}\label{eqn:ConnectionOnSpinorBasis}
	\nabla\psi =\sum_h\biggl( da_h\otimes u_h+ \frac12 \sum_{k<j}a_h\theta_{kj}\otimes (e^j)^\sharp \cdot e_k\cdot u_h\biggr);
\end{equation}
see \cite[Theorem~II.4.14]{Lawson_Michelsohn_1989} for the  Riemannian case and \cite[Eq.~(4)]{BarGauduchonMoroianu} for general signature. In proofs, we will avoid referencing the connection form and use the equivalent formula
\begin{equation}\label{eqn:ConnectionOnSpinorNabla}
	\nabla_{X}\psi =\sum_h (\Lie_{X}a_h) u_h + \frac14\sum_{h,j}a_h (e^j)^\sharp \cdot (\nabla_X e_j)\cdot u_h.
\end{equation}
Similarly, the curvature tensor $R$ acts on spinors as
\begin{equation}
	\label{eqn:RXYpsi}
	R(X,Y)\psi=\frac14\sum_{j}(e^j)^\sharp\cdot R(X,Y)e_j\cdot\psi;
\end{equation}
see \cite[Theorem~II.4.15]{Lawson_Michelsohn_1989}.

This paper is concerned with Killing spinors, i.e. spinors satisfying~\eqref{eqn:Killingeqn}. In this case, the curvature tensor acts as
\begin{equation}
 \label{eqn:curvatureonkillingspinor}
 R(X,Y)\psi=2\lambda^2(YX+g(X,Y))\psi,
\end{equation}
see e.g. \cite[Lemma 4.1]{Conti_Segnan_2024}. The simplest examples of pseudo-Riemannian manifolds admitting a Killing spinor are those with constant curvature, in particular hyperbolic space. We will consider it in its half-space incarnation, defined as
\begin{equation}
	\label{eqn:hyperbolicspace_as_a_manifold}
	H^\epsilon_a=\{(x_1,\dotsc, x_{n},t)\in\R^{n+1}, t>0\}
\end{equation}
endowed with the metric
\begin{equation}
	\label{eqn:hyperbolichalfspace}
	\frac{1}{a^2t^2}(\epsilon_1 dx^1\otimes dx^1 + \dotsc +\epsilon_{n} dx^{n}\otimes dx^{n} +\epsilon_{n+1} dt\otimes dt),
\end{equation}
where $\epsilon=(\epsilon_1,\dotsc, \epsilon_{n+1})\in\{-1,1\}^{n+1}$, $a>0$. If all $\epsilon_i$ are positive this is ordinary hyperbolic space; otherwise, it is isometric to the complement of a degenerate hyperplane in hyperbolic space (see~\cite{SeppiTrebeschi}), and in particular not complete. For $H^\epsilon_a$, every element of the spin representation extends to a Killing spinor (\cite[Theorem~1]{Baum_imaginary}).

For the Riemannian hyperbolic space, the constant $\lambda$ in~\eqref{eqn:Killingeqn} is imaginary, due to~\eqref{eqn:scalarifKilling}.
The Riemannian hyperbolic space is of course homogeneous; it turns out that it is the only homogeneous Riemannian manifold admitting a Killing spinor with $\lambda$ imaginary. This fact can be proved in two ways. Here, we give a proof based on \cite[Theorems 2-3]{Baum_imaginary} and standard techniques of Riemannian geometry.

We begin with an explicit description of Killing fields on a specific type of warped product.
\begin{lemma}
	\label{lemma:WarpKill}
	Let $(F,g)$ be a Riemannian manifold and let $(F\times \R,\tilde g)$ be the warped product, $\tilde g=e^{-4\mu t}g+dt^2$, with $\mu$ a nonzero constant. A vector field on $F\times\R$ is Killing for $\tilde g$ if and only if it takes the form
	\[\tilde Y=X - \frac1{4\mu}e^{4\mu t}\grad_g f+ f\D t,\]
	where:
	\begin{itemize}
		\item $f\colon F\to\R$ is a function such that $\grad_g f$ is parallel;
		\item $X$ is a vector field on $F$ such that $\Lie_{X}g=4\mu fg$.
	\end{itemize}
	Both $X$ and $f$ are uniquely determined by $\tilde Y$.
\end{lemma}
\begin{proof}
	Let $\tilde Y=X_t+ \tilde f\D t$, where $X_t$ is a one-parameter family of vector fields on $F$ and $\tilde f$ a function on $F\times\R$. We indicate by $\tilde\Lie_{\tilde Y}$ the Lie derivative in $F\times\R$ and by $\Lie_{X_t}$ the Lie derivative taken along each hypersurface $F\times\{t\}$. We will use the notations $\tilde f'$ and $X_t'$ to denote derivative in the $t$ variable; since $X_t$ can also be viewed as a vector field on $F\times\R$, we can write $[\D t,X_t]=X_t'$.

	Let $Z$ be a vector field on $F$, extended to $F\times\R$ in the obvious way, so that $[\D t,Z]=0$.  Thus,
	\[\left[\tilde Y,\D{t}\right]=-X_t'-f'\D{t}, \qquad \left[\tilde Y,Z\right]=\Lie_{X_t}Z -(Zf)\D{t}.\]
	Now suppose that $\tilde Y$ is Killing, and $Z,W$ are vector fields on $F$ extended to $F\times\R$ as above. Then
	\[0=(\tilde\Lie_{\tilde Y}\tilde g)\left(\D t, \D t\right)=
		\tilde\Lie_{\tilde Y}\left(\tilde g\left(\D{t},\D{t}\right)\right)-2\tilde g\left(\left[\tilde Y,\D{t}\right],\D t\right)
		=2\tilde f';\]
	thus, $\tilde f$ is independent of $t$ and can be identified with a function $f\colon F\to\R$, as in the statement. Furthermore,
	\begin{multline*}
		0=(\tilde\Lie_{\tilde Y}\tilde g)(Z,W)=\tilde\Lie_{\tilde Y}(e^{-4\mu t}g(Z,W))-e^{-4\mu t}g([\tilde Y,Z],W)-e^{-4\mu t}g(Z,[\tilde Y,W])\\
		=-4\mu fe^{-4\mu t}g(Z,W)+ e^{-4\mu t}(\Lie_{X_t} g)(Z,W),
	\end{multline*}
	and therefore
	\begin{equation}
		\label{eqn:Xtkilling}
		\Lie_{X_t} g=4\mu f g, \quad \text{ for all } t.
	\end{equation}
	Moreover,
	\begin{multline*}
		0=(\tilde\Lie_{\tilde Y}\tilde g)\left(Z, \D t\right)
		=-\tilde g\left([\tilde Y,Z],\D{t}\right)-\tilde g\left(Z,\left[\tilde Y,\D{t}\right]\right)\\
		=-\tilde g\left(-(Zf)\D t,\D t\right) - \tilde g (Z,-X_t')
		=Zf + e^{-4\mu t}g(X_t',Z),
	\end{multline*}
	i.e. $\grad_g f = -e^{-4\mu t}X_t'$; integrating in $t$ yields
	\[X_t=X-\frac1{4\mu} e^{4\mu t}\grad_g f,\]
	where $X$ is a vector field satisfying~\eqref{eqn:Xtkilling}.

	Therefore, $\tilde Y$ has the form in the statement; it only remains to show that $\grad_g f$ is parallel. Substituting in~\eqref{eqn:Xtkilling}, we see that $\Lie_{\grad_g f}g=0$, i.e.  $\grad_g f$ is a Killing vector field. Therefore, $\nabla \grad_gf$ is skew-symmetric; since it coincides with the Hessian of $f$, which is symmetric, we conclude that $\grad_g f$ is parallel.

	The fact that any $\tilde Y$ as in the statement is Killing follows from the same computations. Finally, $f=dt(\tilde Y)$ is clearly determined by $\tilde Y$, and so must be $X$.
\end{proof}
Restricting to homogeneous warped products where in addition $g$ is Ricci-flat, we obtain the following characterization:
\begin{proposition}
	\label{prop:homogeneouswarped}
	Let $(F,g)$ be a Ricci-flat Riemannian manifold and let $\tilde g=e^{-4\mu t}g+dt^2$ be a warped product metric on $F\times\R$, with $\mu$ a nonzero constant. If $(F\times \R,\tilde g)$ is homogeneous, then $F$ is flat and $F\times\R$ is isometric to hyperbolic space.
\end{proposition}
\begin{proof}
	We first prove that $F$ is flat. Let $\lie{iso}(F\times\R)$ be the Lie algebra of Killing vector fields on $(F\times\R,\tilde g)$. By Lemma~\ref{lemma:WarpKill}, any Killing vector field determines a parallel vector field $\grad_g f$, i.e. we have an $\R$-linear map
	\begin{equation}
		\label{eqn:linearKill}
		\lie{iso}(F\times\R)\to \mathfrak{X}(F), \quad \tilde Y\mapsto \grad_g f=\grad_g (dt(\tilde Y)),
	\end{equation}
	whose image $V$ consists of parallel vector fields. Choose Killing vector fields
	\[\tilde Y_i=X_i-\frac1{4\mu}e^{4\mu t}\grad_g f_i + f_i\D t, \quad i=1,\dotsc, k,\]
	such that $\grad_g f_1,\dotsc, \grad_g f_k$ is a basis of $V$. Notice that each $\grad_g f$ in $V$ is a complete vector field, because its integral lines are geodesics and $F$ is complete. These integral lines are mapped to straight lines through the origin under the map $(f_1,\dotsc, f_k)\colon F\to\R^k$; in particular, this map is surjective, and also a submersion, as the gradients are independent. Since the gradients are parallel, $F$ is a Riemannian product $\check F\times\R^k$, where $\R^k$ has the Euclidean metric.

	By Lemma~\ref{lemma:WarpKill}, the kernel of~\eqref{eqn:linearKill} consists of vector fields of the form $X+c\D t$
	with $c$ constant and $\Lie_{X} g=4\mu cg$,
	and the space $\lie{iso}(F)$ of  Killing vector fields on $F$ can be identified with the subspace of the kernel for which  $c$ is zero. Setting
	\[W=\Span{\tilde Y_1,\dotsc, \tilde Y_k}\oplus\lie{iso}(F),\]
	we see that $W$ has codimension at most one in
	\[\lie{iso}(F\times\R)=\Span{\tilde Y_1,\dotsc, \tilde Y_k}\oplus\left\{X+c\D t\mid \Lie_X g=4\mu c g, c\in\R\right\}.\]
	Fix $p=(x,\underline{0},0)\in \check{F}\times \R^k\times\R$; evaluating at $p$ gives a linear map
	\begin{equation*}
		\label{eqn:evaluateKillingatp}
		\rho\colon \lie{iso}(F\times\R)\to T_p(F\times \R), \quad \tilde Y\mapsto \tilde Y_p,
	\end{equation*}
	which is surjective because $F\times\R$ is homogeneous. Then $\rho(W)$ has codimension at most one in $T_p(F\times\R)$. For every $i=1,\dotsc, k$, $f_i=0$ at $(x,\underline 0)$, and $\tilde Y_i$ is tangent to $F\times\{0\}$ at $p$;
	hence $\rho(W)\subseteq T_{x,\underline 0}F$. Since $\rho(W)$ has codimension at most one in $T_p(F\times\R)$, we conclude $\rho(W)=T_{x,\underline 0}F$. Hence, on $\check{F}\times\{\underline 0\}$, any $Z$ in $W$ is tangent to $F$ and satisfies $\Lie_{Z} g=4\mu fg=0$; restricting to $\check{F}\times\{\underline 0\}$ and projecting, $Z$ determines a Killing vector field on $\check{F}$.  Thus, $T_x\check{F}$ is spanned by Killing vector fields of $\check{F}$. This shows that $\check{F}$ is locally homogeneous; passing to the universal cover, we obtain a homogeneous  Ricci-flat Riemannian manifold, which is necessarily flat, as proved in~\cite{Alekseevskii_Kimelfeld_1975}.

	Since  $F$ is flat, the metric $\tilde g$ is locally of the form $e^{-4\mu t}(dx_1^2+\dotsc + dx_n^2)+dt^2$, which up to a change of coordinate is equivalent to~\eqref{eqn:hyperbolichalfspace}. Thus, $F\times\R$ has constant negative curvature. Since the only homogeneous Riemannian manifolds of constant negative curvature are the hyperbolic spaces (\cite[Theorem~2.7.1]{Wolf:SpacesOfConstantCurvature}), the statement is proved.
\end{proof}

Combining Proposition~\ref{prop:homogeneouswarped} with~\cite{Baum_imaginary} we obtain:
\begin{corollary}
	\label{cor:conbaum}
	The only homogeneous Riemannian manifold carrying a Killing spinor with $\lambda$ imaginary is hyperbolic space.
\end{corollary}
\begin{proof}
	Homogeneous Riemannian manifolds are complete. By \cite[Theorems 2-3]{Baum_imaginary}, a complete Riemannian manifold with a Killing spinor with $\lambda=i\mu$, $\mu\in\R\setminus\{0\}$ is either isometric to hyperbolic space or to a warped product of the form $\tilde g=e^{-4\mu t}g+dt^2$ on the manifold $F\times\R$, where $(F,g)$ is a complete, Ricci-flat Riemannian manifold admitting a parallel spinor. By Proposition~\ref{prop:homogeneouswarped}, the second case reduces to the first for homogeneous metrics.
\end{proof}

\section{Standard and pseudo-Iwasawa solvmanifolds}
\label{sec:iwasawa}
In this preparatory section, we recall the notions of standard and pseudo-Iwasawa decompositions on a solvmanifold and the relation between Einstein solvmanifolds and nilsolitons, following~\cite{ContiRossi_IndeNilsEinsSov}; we also give some useful formulae for the Levi-Civita connection on a pseudo-Iwasawa solvmanifold and its curvature.

From now on, we will consider pseudo-Riemannian solvmanifolds, namely solvable Lie groups endowed with a left-invariant metric. We will denote by $\tilde G$ the solvable Lie group, by $\tilde\g$ its Lie algebra, and by $\tilde g$ the nondegenerate scalar product on $\tilde\g$ which determines the metric by left translation. The Levi-Civita connection and its curvature can then be viewed as linear maps on the Lie algebra $\tilde\g$, whereas spinors can be viewed as smooth maps from $\tilde G$ to the spinor representation. We stress the fact that spinors will not be assumed to be left-invariant.

Recall from~\cite{ContiRossi_IndeNilsEinsSov} that a solvmanifold is said to admit a \emph{standard decomposition} if its Lie algebra $\tilde \g$ decomposes as an orthogonal sum $\tilde\g=\g\rtimes\lie a$,  where $\g$ is a nilpotent ideal and $\lie a$ an abelian subalgebra; this condition depends not only on the Lie algebra structure, but also on the metric assigned to the solvmanifold. Given a standard decomposition, we will write the metric in the form
\begin{equation}
	\label{eqn:standardmetric}
	\tilde g=g+\sum_\alpha \epsilon_\alpha e^\alpha\otimes e^\alpha,
\end{equation}
where $\{e_\alpha\}$ is an orthonormal basis of $\lie a$, $\{e^\alpha\}$ is the dual basis, $\epsilon_\alpha=\pm1$, and $g$ is a metric on $\g$. Similarly, we will write $g=\sum_i \epsilon_i e^i\otimes e^i$; in general, we will use indices $i,j$  for elements of $\g$, and $\alpha,\beta$ for elements of $\lie a$. We will set
\[\phi_\alpha=-\ad e_\alpha.\]

In the Riemannian case, Einstein solvmanifolds always have a standard decomposition. We illustrate this in the following remarks; the reader may consult \cite{Jablonski:Survey} for a general survey on homogeneous Riemannian Einstein manifolds.
\begin{remark}
	\label{remark:stdexistsnonzero}
	In the Riemannian case (see e.g.~\cite{Heber_1998,Lauret:Nilsolitons}), a solvmanifold is said to be standard if the Lie algebra $\tilde\g$ decomposes as an orthogonal sum $\tilde\g'\rtimes (\tilde\g')^\perp$, where $\tilde\g'=[\tilde\g,\tilde\g]$ is the derived algebra and its orthogonal complement $(\tilde\g')^\perp$ is abelian; this is a standard decomposition in the sense above, because the derived algebra of a solvable Lie algebra is nilpotent. The definition of standard decomposition is more general, because it does not impose $\g=\tilde\g'$; however, the two notions agree for Riemannian Einstein solvmanifolds of nonzero scalar curvature. Indeed, by~\cite{Lauret_2010}, a Riemannian Einstein solvmanifold with nonzero scalar curvature is standard, i.e. admits a standard decomposition $\tilde\g'\rtimes(\tilde\g')^\perp$. This is the only standard decomposition: given another, say  $\g\rtimes\lie a$, on the one hand
	\[\tilde \g'=[\g\rtimes\lie a,\g\rtimes\lie a]\subseteq \lie g,\]
	and on the other hand for any $x$ in $\g\cap(\tilde\g')^\perp$, $\ad x$ is simultaneously normal (by \cite[Theorem~B]{Heber_1998}) and nilpotent (because $\g$ is nilpotent), hence zero; this implies that $x$ is in the center. Since $x$ is orthogonal to the derived algebra, $\tilde \g$ splits as an orthogonal direct product $\Span{x}\oplus\Span{x}^\perp$, so if $x$ is nonzero the solvmanifold is a Riemannian product with a flat factor, which is absurd.
	Thus, $\g\cap(\tilde\g')^\perp$ is zero, and $\g$ cannot be larger than $\tilde\g'$, i.e. the standard decomposition is unique.
\end{remark}

\begin{remark}
	\label{remark:stdexistszero}
	A Ricci-flat Riemannian solvmanifold is flat by~\cite{Alekseevskii_Kimelfeld_1975}, and by~\cite{Milnor_1976} has a standard decomposition with $\g$ commutative and $\lie a$ acting by skew-symmetric transformations. However, the standard decomposition is not unique in this case, as any abelian direct factor in $\tilde\g$ can be assigned to either $\g$ or $\lie a$.
\end{remark}

\begin{remark}
	On Einstein solvmanifolds of indefinite signature, standard decompositions may not exist. For instance, on solvable Lie algebras $\tilde\g$ such that the nilradical $\lie n$ coincides with the derived algebra $\tilde \g'=[\tilde\g,\tilde\g]$, an Einstein metric may exist for which $\lie n=\tilde \g'$ is degenerate, preventing the existence of a standard decomposition (see \cite[Example~1.6]{ContiRossi_IndeNilsEinsSov}). See also~\cite{Rossi_2025} for an example of an Einstein solvmanifold where $\lie n=\tilde\g'$ is nondegenerate but no standard decomposition exists. Moreover, standard decompositions may not be unique, not only in the flat case, but also for Ricci-flat, non-flat metrics (\cite[Example~4.16]{ContiRossi_IndeNilsEinsSov}) and Einstein metrics of nonzero scalar curvature (\cite[Example~3.3]{ContiRossi_IndeNilsEinsSov}).
\end{remark}

The Ricci tensor of a solvmanifold with a standard decomposition has a simple expression, see e.g.~\cite[Proposition~1.10]{ContiRossi_IndeNilsEinsSov}; we will use the following adapted version:
\begin{proposition}
	\label{prop:ricandrictilde}
	Let $\g$ be a nilpotent Lie algebra, let $\lie a$ be an abelian Lie algebra, and let $\tilde g$  be a metric of the form~\eqref{eqn:standardmetric} on $\tilde\g=\g\rtimes\lie a$. The Ricci tensors of the metrics $g$, $\tilde g$ are related by
	\begin{gather*}
		\widetilde\ric(v,w)=\ric(v,w)+\sum_\alpha\frac12\epsilon_\alpha g([\phi_\alpha,\phi_\alpha^*](v),w) - \epsilon_\alpha g((\phi_\alpha)^s(v),w)\Tr\phi_\alpha,\\
		\widetilde\ric(v,e_\alpha)=\frac12 \Tr(\ad v\circ \phi_\alpha^*),\qquad \widetilde \ric(e_\alpha,e_\beta)=-\Tr((\phi_\alpha)^s\circ\phi_\beta),
	\end{gather*}
	for any $v, w$ in $\g$.
\end{proposition}
\begin{proof}
	For endomorphisms $f_1,f_2$ of $\tilde \g$, a direct computation shows that
	\[\langle f_1, f_2\rangle= \Tr( f_1 f_2^*)= \Tr( f_1 (f_2^s-f_2^a)),\]
	where $f_2^a=\frac12(f_2-f_2^*)$ is the antisymmetric part of $f_2$. The statement now follows by applying the above to \cite[Proposition~1.10]{ContiRossi_IndeNilsEinsSov}.
\end{proof}
Proposition~\ref{prop:ricandrictilde} will be used in the proof of the main theorem, but it can also be applied to illustrate the correspondence between pseudo-Iwasawa solvmanifolds and nilsolitons.

A standard decomposition is called \emph{pseudo-Iwasawa} if $\ad X$ is symmetric for all $X\in\lie a$; we will say that a solvmanifold is pseudo-Iwasawa if it admits a pseudo-Iwasawa standard decomposition. These definitions generalize one given in~\cite{Heber_1998}. In our language, \cite[Theorem~B]{Heber_1998} shows that standard Riemannian Einstein solvmanifolds are  pseudo-Iwasawa up to isometry; the isometry is at the level of Riemannian manifolds, and does not imply that the corresponding Lie algebras are isomorphic.

\begin{remark}
	\label{remark:RiemannianSolv}
	Combining \cite[Theorem~B]{Heber_1998} with the facts that Riemannian Einstein solvmanifolds have standard decompositions (\cite{Lauret_2010}, see Remark~\ref{remark:stdexistsnonzero}) and that homogeneous Einstein Riemannian manifolds of negative scalar curvature are solvmanifolds (\cite{Bohm_Lafuente_2023}), we see that every homogeneous Einstein Riemannian manifold of negative scalar curvature is isometric to a solvmanifold with a standard decomposition of pseudo-Iwasawa type.
\end{remark}

Applying Proposition~\ref{prop:ricandrictilde} to a pseudo-Iwasawa Einstein solvmanifold, the  Ricci operator on $\g$ has the form
\begin{equation}
	\label{eqn:nilsoliton}
	\Ric=\lambda \id + D,
\end{equation}
where $\lambda$ is the Einstein constant and $D=-\widetilde\ad \bigl(\sum_\alpha\epsilon_\alpha (\Tr \phi_\alpha)e_\alpha\bigr)$. In general, a nilpotent Lie group with a left-invariant metric satisfying~\eqref{eqn:nilsoliton} for some derivation $D$ is called a nilsoliton; thus, any pseudo-Iwasawa Einstein solvmanifold determines a nilsoliton.

Conversely, a nilsoliton with $\Tr D^2\neq0$ determines an Einstein metric on the pseudo-Iwasawa solvmanifold $\g\rtimes_D\R$  (see~\cite{Lauret:Nilsolitons} and \cite[Theorem~3.9]{ContiRossi_IndeNilsEinsSov}, respectively for the Riemannian and indefinite case).

It follows that, in the Riemannian case, all homogeneous Einstein manifolds of negative scalar curvature can be reduced to extensions of a nilsoliton (see~\cite{Jablonski:Survey}). However, in the indefinite case, there are more Einstein solvmanifolds than those obtained by extending a nilsoliton (see e.g.~\cite{Conti_Rossi_Segnan_2023,Conti_Rossi_SegnanDalmasso_2024}).

\begin{remark}
	Nilsolitons owe their name to the fact that a left-invariant Riemannian metric on a nilpotent group is a Ricci soliton if and only if~\eqref{eqn:nilsoliton} holds for some derivation $D$. This result, initially stated in~\cite{Lauret:Nilsolitons}, was proved in~\cite{Jablonski_2014}.
\end{remark}

We conclude this section with some basic formulae that will be used in subsequent sections.
\begin{lemma}
	\label{lemma:equazione8}
	The Levi-Civita connection $\widetilde\nabla$ of $\tilde g$ can be expressed in terms of the Levi-Civita connection $\nabla$ of $g$ by
	\[
		\widetilde\nabla_{e_\alpha} x=0,\qquad  \widetilde\nabla_w{e_\alpha}=\phi_\alpha w,\qquad
		\widetilde \nabla_w v= \nabla_w v-\sum_\alpha g(\phi_\alpha w,v)\epsilon_\alpha e_\alpha,
	\]
	where $v,w$ are in $\g$ and $x$ is in $\tilde \g$. In particular, the following hold:
\[\tilde R(v,e_\alpha)w=-\widetilde\nabla_{\phi_\alpha v}w,\quad
\tilde R(v,e_\alpha)e_\beta =-\phi_\beta\phi_\alpha v.\]
\end{lemma}
\begin{proof}
	For any $x,y\in\tilde\g$, the Koszul formula can be written as
	\[\widetilde\nabla_x y=-\widetilde\ad(y)^sx-\frac12\widetilde\ad (x)^*y.\]
	For $v\in\g$, we have
	\[\widetilde\ad(v) = \ad (v) + \sum_\alpha e^\alpha\otimes\phi_\alpha v, \qquad \widetilde \ad (v)^* = \ad (v)^*+ \sum_\alpha  (\phi_\alpha v)^\flat \otimes \epsilon_\alpha e_\alpha;\]
	in particular,
	\[\widetilde\ad(v)^se_\alpha=\frac12(\widetilde\ad(v)e_\alpha + \widetilde\ad(v)^*e_\alpha)=\frac12\phi_\alpha v,\]
	where $e_\alpha$ is any element in the fixed basis of $\lie a$. Similarly, we have
	\[\widetilde\ad (e_\alpha) = -\phi_\alpha=\widetilde\ad (e_\alpha)^*.\]
	Hence, for $v,w\in\g$ we get
	\begin{gather*}
		\widetilde \nabla_{e_\alpha} v=
		-\widetilde\ad(v)^se_\alpha-\frac12(\widetilde\ad (e_\alpha))^*v=-\frac12\phi_\alpha v+\frac12\phi_\alpha v=0,\\
		\widetilde\nabla_w e_\alpha=-\widetilde\ad (e_\alpha)^s w - \frac12\widetilde\ad(w)^*e_\alpha=\phi_\alpha w,\\
		\widetilde \nabla_w v=-\widetilde\ad(v)^sw-\frac12\widetilde\ad (w)^*v= \nabla_w v-\sum_\alpha g(\phi_\alpha w,v)\epsilon_\alpha e_\alpha,\\
		\widetilde \nabla_{e_\alpha} e_\beta = 0.
	\end{gather*}

In particular, for $v,w\in\g$, the curvature satisfies
	\begin{align*}
		\tilde R(v,e_\alpha)w       & =\widetilde\nabla_{v}\widetilde\nabla_{e_\alpha} w-\widetilde\nabla_{e_\alpha}\widetilde\nabla_{v}w-\widetilde\nabla_{[v,e_\alpha]}w=-\widetilde\nabla_{\phi_\alpha v}w,                                                  \\
		\tilde R(v,e_\alpha)e_\beta & =\widetilde\nabla_{v}\widetilde\nabla_{e_\alpha} e_\beta-\widetilde\nabla_{e_\alpha}\widetilde\nabla_{v}e_\beta-\widetilde\nabla_{[v,e_\alpha]}e_\beta=-\widetilde\nabla_{\phi_\alpha v}e_\beta=-\phi_\beta\phi_\alpha v.\qedhere
	\end{align*}
	\end{proof}

\section{Pseudo-Iwasawa solvmanifolds with a Killing spinor}
\label{sec:construction}
In this section we construct a class of pseudo-Iwasawa solvmanifolds $\tilde G^{a,k}=G\rtimes\R$ admitting Killing spinors. This family includes the hyperbolic half-space, corresponding to the case $k=0$.

The underlying Lie algebra is the semidirect product $\g\rtimes_{\phi_0}\Span{e_0}$, where $\g$ is abelian and
\[\phi_0=a\id+ f,\]
where $a\in\R$, $k\in\N$ and $f$ is a fixed endomorphism of rank $k$ with $f^2=0$. We write the general element of $\g\cong\R^n$ as $(x,y,z)\in\R^k\times\R^{n-2k}\times\R^k$, and choose  the endomorphism $f$ and  metric $g$ of the form
\begin{equation}
 \label{eqn:canonicalnilpotent}
 f=\begin{pmatrix} 0 & 0 & I_k \\ 0 & 0 & 0 \\ 0 & 0 & 0 \end{pmatrix}, \qquad
g_{\delta,\epsilon}=\begin{pmatrix} 0 & 0 & \Delta \\ 0 & E & 0 \\ \Delta & 0 & 0 \end{pmatrix},
\end{equation}
where $I_k$ is the identity matrix and $\Delta,E$ are determined by $\delta=(\delta_1,\dotsc,\delta_k)\in\{\pm1\}^k$, $\epsilon=(\epsilon_1,\dotsc, \epsilon_{n-2k})\in\{\pm 1\}^{n-2k}$ via
\[\Delta= \diag(\delta_1,\dotsc, \delta_k), \quad E=\diag(\epsilon_1,\dotsc, \epsilon_{n-2k}).\]
We denote by $\tilde G^{a,k}$ the corresponding connected simply connected Lie group. When referring to $\tilde G^{a,k}$ as a solvmanifold, we will implicitly fix the metric
\[\tilde g_{\epsilon,\delta,\tau}=g_{\delta,\epsilon}+\epsilon_0 e^0\otimes e^0=g_{\delta,\epsilon}+\tau^2 e^0\otimes e^0,\]
where  $\tau\in\{1,i\}$ determines $\epsilon_0=\tau^2$.  The product law can be written as
\[(x,y,z,t)(x',y',z',t')=(e^{at}(x'+tz')+x,e^{at}y'+y,e^{at}z'+z,t+t'),\]
so that a left-invariant frame is given by
\[\left(e^{at}\D{x_i}, e^{at}\D{y_j}, e^{at}(t\D{x_i}+\D{z_i}), \D{t}\right).\]
The metric is then expressed as
\[\tilde g_{\delta,\epsilon
,\tau}=e^{-2at}\biggl(\sum_{i=1}^k \delta_i(dx^i-tdz^i)\odot dz^i + \sum_{j=1}^{n-2k} \epsilon_j dy^j\otimes dy^j \biggr)+ \epsilon_0dt^2.\]

We will allow $k=0$, in which case $\phi_0=a\id$, the variables $x,z$ do not appear, and  $\tilde G^{a,0}$ can be identified with the hyperbolic half-space $H^\epsilon_a$  by writing $s=e^{at}$ and rescaling the coordinates $y$ to obtain the metric \eqref{eqn:hyperbolichalfspace}.

If $a=0$, the metric takes the form
\begin{equation}
 \label{eqn:brinkmannfaciao}
 \tilde g_{\delta,\epsilon,\tau}=\sum_{i=1}^k \delta_i(dx^i-tdz^i)\odot dz^i + \sum_{j=1}^{n-2k} \epsilon_j dy^j\otimes dy^j +\epsilon_0 dt^2,
\end{equation}
and Lemma~\ref{lemma:equazione8} gives  $\widetilde\nabla v=0$ for all $v\in\ker f$. In particular, $\im f$ is a parallel isotropic distribution and the holonomy algebra is contained in the abelian Lie algebra
\[\Span{dz^i\otimes \D{x_j}}.\]
\begin{remark}
 \label{rmk:brinkmann}
If  $\epsilon_j=1$ for all $j$  and $k=1$, \eqref{eqn:brinkmannfaciao} gives a Lorentzian manifold with a parallel light-like vector field, also called a \emph{Brinkmann wave}; due the form of the holonomy algebra, these Lorentzian metrics fall into the  class characterized in \cite[Theorem 4.1]{Leistner_2002}.
\end{remark}
\begin{remark}
\label{remark:oneigenvalueimpliesGak}
If $\phi_0\colon\R^{p,q}\to\R^{p,q}$ is symmetric and only has one eigenvalue $a\in\R$, we can reduce to the situation \eqref{eqn:canonicalnilpotent} as follows. Write $\phi_0=a\id +f$; then $f^2=0$. Choose a decomposition
\[\R^{p,q}=\ker f\oplus U;\]
since $(\ker f)^\perp=\im f\subset \ker f$, $g$ induces a pairing between $U$ and $\im f$, and we can assume that $U$ is isotropic.  Write
\[\ker f=\im f\oplus W,\]
where $W=\ker f\cap U^{\perp}$; then $W\cap W^{\perp}=\{0\}$, so $W$ admits an orthonormal basis $\{w_i\}$. The symmetric bilinear form on $U$ given by $\langle v,w\rangle = g(fv,w)$ is nondegenerate; let $\{u_i\}$ be a $\langle,\rangle$-orthonormal basis, and let $\{v_i\}$ be the basis of $\im f$ obtained by setting $f(u_i)=v_i$. Adjoining the bases $\{v_i\}$, $\{w_i\}$ and $\{u_i\}$ we obtain a basis which puts $f,g$ in the form~\eqref{eqn:canonicalnilpotent}.
\end{remark}

\begin{remark}
The solvmanifolds $\tilde G^{a,k}$ are algebraic Ricci solitons, i.e. they satisfy \eqref{eqn:nilsoliton} for some derivation $D$, in this case a multiple of $f$. More precisely, Proposition~\ref{prop:ricandrictilde} gives
\begin{multline*}
\widetilde\Ric = -\epsilon_0\Tr(\phi_0) \phi_0 - \epsilon_0\Tr((\phi_0)^2)e^0\otimes e_0\\
=-\epsilon_0 na(a\id_\g + f) -\epsilon_0na^2e^0\otimes e_0
=-\epsilon_0na(a\id_{\tilde \g} + f).
\end{multline*}
In particular, the metric is Einstein only for $k=0$, namely in the case of the hyperbolic half-space.
\end{remark}

We will show that the solvmanifolds $\tilde G^{a,k}$ admit Killing spinors. This will require a technical lemma. Given $W\subset\R^{p,q}$, we will denote by $\Sigma_{p,q}^W$ the space of spinors $\psi\in\Sigma_{p,q}$ annihilated by $W$, i.e. such that  $w\cdot\psi=0$ for all $w$ in $W$. Notice that $\Sigma_{p,q}^W$ can only be nontrivial when $W$ is isotropic, implying in particular that $\dim W\leq p,q$.
\begin{lemma}
\label{lemma:countspinors}
Let $W\subset\R^{p,q}$ be an isotropic space of dimension $k\geq0$. Then
\begin{itemize}
\item $\Sigma_{p,q}^W$  has dimension $2^{[(p+q)/2]-k}$;
\item if $v$ is a nonisotropic vector in $W^\perp$, the map
\begin{equation}
 \label{eqn:annihilatedtoannihilated}
 \Sigma_{p,q}^W\to \Sigma_{p,q}^W, \quad \psi\mapsto v\cdot\psi
\end{equation}
is a multiple of the identity if $p+q=2k+1$ and has two eigenspaces of dimension $2^{[(p+q)/2]-k-1}$ otherwise.
\end{itemize}
\end{lemma}
\begin{proof}
Denote by $\langle,\rangle$ the complex bilinear form on $\C^{p+q}$ induced by complexifying the scalar product on $\R^{p,q}$. Then the complexified subspace $W^\C\subset\C^{p+q}$ is isotropic with respect to $\langle,\rangle$. We will distinguish several cases and prove the two items in each.

Suppose first that $p=q=k$, and denote by $\CCl(\C^{2k})$ the complex Clifford algebra induced by $\langle,\rangle$. The complex spinor representation of $\R^{k,k}$, which we denote $\Sigma_{k,k}$, is by construction a representation of $\CCl(\C^{2k})$, and a spinor is annihilated by $W$ if and only if it is annihilated by $W^\C$. Since $W^\C\subset\C^{2k}$ is maximal isotropic, it is known that there exists a complex spinor $\psi\in\Sigma_{k,k}$ annihilated by $W^\C$; the spinor $\psi$ is unique up to complex multiple (\cite[Proposition 9.7]{Lawson_Michelsohn_1989}). Thus, $\Sigma_{k,k}^W$ has dimension one, and the first item is proved in this case. By a count of dimensions, $W=W^\perp$, so the second part is an empty statement.

In general, we fix a $\langle,\rangle$-orthogonal decomposition,
\[\C^{p+q}=U\oplus V,\]
where $\dim U=2k$ and $W^\C$ is maximal isotropic in $U$. Notice that the generic element orthogonal to $W^\C$ lies in $W^\C\oplus V$; since $W$ acts trivially on $\Sigma_{p,q}^W$, there will be no loss of generality in assuming that $v$ lies in $V$. Notice also that the map~\eqref{eqn:annihilatedtoannihilated} is well defined because $v$ anticommutes with $W$ in the Clifford algebra.

If $V$ has dimension one, then (see e.g. \cite[Equation~(2)]{Baum_Kath_1999})
\[\CCl(\C^{p+q})=\CCl(U\oplus \C)\cong \CCl(U)\oplus\CCl(U).\]
The spinor representations associated to $U$ and to $\C^{p+q}=U\oplus \C$ correspond to the same representation of $\CCl(U)$. Arguing as above, we see that $\Sigma_{p,q}^W$ is one dimensional. For dimensional reason, a vector $v$ in the complementary $\C$ can only act on $\Sigma_{p,q}^W$ as a multiple of the identity.

If $V$ has dimension $2m$, then we have an isomorphism (\cite[Theorem 4.3]{Lawson_Michelsohn_1989})
\[\CCl(U\oplus V)\cong\CCl(U)\otimes\CCl(\C^2)^{\otimes m},\]
where $\CCl(\C^2)$ is the algebra of complex two-by-two matrices; accordingly, 
\[\Sigma_{p,q}\cong\Sigma_{2k}\otimes (\C^2)^{\otimes m},\]
and thus $\Sigma_{p,q}^W=\Sigma_{2k}^W\otimes (\C^2)^{\otimes m}$ has dimension $2^m$. The same conclusion is obtained for $V$ of dimension $2m+1$, writing $\C^{p+q}=(U\oplus\C^{2m})\oplus\C$ and combining the two arguments.

It remains to show that the second item holds when $V$ has dimension greater than one. In this case, the map \eqref{eqn:annihilatedtoannihilated} is diagonalizable because it squares to a multiple of the identity. Choosing a unit $w$ in $V$ orthogonal to $v$, Clifford multiplication by $w$ interchanges the eigenspaces of~\eqref{eqn:annihilatedtoannihilated}, because $v\cdot w + w\cdot v=0$, so the eigenspaces have the same dimension.
\end{proof}

\begin{proposition}
\label{prop:counterexamplestomaintheorem}
For any $k\geq0$, $\tau\in\{1,i\}$, $a\neq0$,  define the subspace
\[\Sigma_f^\pm=\{\psi\in\Sigma\mid f(v)\cdot\psi=0\, \forall v\in\R^{n},  \psi=\pm i\tau e_0\cdot\psi\}\]
of the fibre at the identity of the spinor bundle of $(\tilde G^{a,k},\tilde g_{\delta,\epsilon,\tau})$. Then
\begin{enumerate}
 \item For any $\psi\in\Sigma^\pm_f$, $e^{at/2}\psi$ is a Killing spinor on $\tilde G^{a,k}$ with Killing constant $\pm i\tau a/2$.
 \item If $n=2k$, the space $\Sigma_f^+\oplus \Sigma_f^{-}$ has dimension one, yielding one Killing spinor with constant equal to $i\tau a/2$ or $-i\tau a/2$ according to whether an even or odd number of the constants $\epsilon_0,\dotsc, \epsilon_k$ are positive.
 \item If $n>2k$, each space $\Sigma^\pm_f$ has dimension $2^{[(n-1)/2]-k}$, yielding $2^{[(n-1)/2]-k}$ linearly independent Killing spinors with Killing constant
 $i\tau a/2$ and as many with Killing constant $-i\tau a/2$.
\end{enumerate}
\end{proposition}
\begin{proof}
%
Recall from Lemma~\ref{lemma:equazione8} that
\[\widetilde\nabla_{e_a}=0, \quad \widetilde\nabla_w e_\alpha=\phi_0(w),\quad \widetilde\nabla_w v=-g(\phi_0(w),v) \epsilon_0 e_0.\]
Let
\[\Sigma_f=\{\psi\in\Sigma \mid f(v)\cdot\psi=0\,\forall v\in\R^{n}\}.\]
By Lemma~\ref{lemma:countspinors}, the linear map
\begin{equation}
 \label{eqn:multfore0}
 \Sigma_f\to\Sigma_f, \quad \sigma\mapsto e_0\cdot\sigma\textbf{};
\end{equation}
is diagonalizable with eigenvalues $\pm i\tau$.

In the particular case where $n=2k$, so that $\Sigma_f$ has dimension one, consider the volume element
\[\omega=e_0\cdot \prod_{i=1}^k \frac12(e_i-e_{k+i})\cdot(e_i+e_{k+i})
=e_0\cdot \prod_{i=1}^k \frac12(2\epsilon_i+e_ie_{k+i}),
\]
which is compatible with the orientation defined by the basis $\{e_0,\dotsc, e_{2k}\}$.
If $\psi$ is in $\Sigma_f$, then
\[\omega\cdot\psi=\epsilon_1\dotsm \epsilon_k e_0\cdot \psi.\]
By convention (see e.g. \cite{BarGauduchonMoroianu}), $\omega$ acts on spinors as multiplication by
\[\tau i^{k+(n+1)(n+2)/2}=\tau i^{k+(2k+1)(k+1)}=\tau i^{2k^2+1}=\tau i(-1)^k.\]
Thus,
\[\pm i\tau e_0\cdot\psi = \pm i\tau (-\epsilon_1)\dotsm (-\epsilon_k)( i\tau) \psi = \pm(-\epsilon_0)(-\epsilon_1)\dotsm (-\epsilon_k)\psi,
\]
and $\Sigma_f$ equals either $\Sigma^+_f$ or $\Sigma^-_f$ according to whether an even or odd number of the $\epsilon_i$ are positive.

We must prove that elements of $\Sigma^\pm_f$ can be extended to Killing spinors with Killing constant $\lambda=\pm i\tau a/2$; the dimensional counts then follow from Lemma~\ref{lemma:countspinors}.

Fix any $\psi\in\Sigma_f$ and extend it by left-invariance; then $\widetilde\nabla_{e_0}\psi=0$ and, for $X$ in $\R^{p,q}$,
\begin{equation}
\label{eqn:derivativeofSigmaf}
\begin{split}
 \widetilde\nabla_{X}\psi &= \frac14\sum_{j}(e^j)^\sharp \cdot (\widetilde\nabla_X e_j)\cdot \psi
+\frac14 \epsilon_0 e_0\cdot  (\widetilde\nabla_X e_0)\cdot \psi\\
&= -\frac14
 \phi_0(X)\cdot \epsilon_0 e_0\cdot \psi+\frac14  \epsilon_0 e_0\cdot \phi_0(X)\cdot\psi
=-\frac {\epsilon_0a}{2}  X\cdot  e_0\cdot\psi
\end{split}
\end{equation}
Now set
\[\Psi=e^{at/2}\psi,\]
and assume that $\psi$ is in $\Sigma^\pm_f$. Then \eqref{eqn:derivativeofSigmaf} implies
\[ \widetilde\nabla_{X}\Psi = e^{at/2}\widetilde\nabla_X\psi
=-\frac {\epsilon_0a}{2}e^{at/2}X\cdot (\mp i\tau^3)\psi= \pm \frac{i\tau a}2 X \cdot\Psi.\]
Moreover,
\[\widetilde\nabla_{e_0}\Psi=\frac{a}2\Psi= \pm \frac{i\tau a}2  e_0\cdot\Psi,\]
showing that $\Psi$ is a Killing spinor with Killing constant $\pm i\tau a/2$.
\end{proof}

\begin{remark}\label{remark:a0givesparallel}
The assumption $a\neq0$ in Proposition~\ref{prop:counterexamplestomaintheorem} is only necessary to obtain Killing spinors, which by our convention exclude parallel spinors. However, the same proof applies also to the case $a=0$, showing that the solvmanifolds $\tilde G^{0,k}$ admit left-invariant parallel spinors.
\end{remark}

\section{Proof of the main theorem}
\label{sec:killingNotIwasawa}
In this section we prove that the solvmanifolds $\tilde G^{a,k}$ constructed in Section~\ref{sec:construction} are the only pseudo-Iwasawa solvmanifolds admitting a Killing spinor.

Throughout this section, we consider spinors which are not necessarily invariant, and therefore take the form
$\psi=\sum_h a_h u_h$, where $a_h\colon\tilde G\to\C$ are smooth functions. For $x\in\tilde{\g}$, we will denote by $\Lie_x a_h$ the Lie derivative of $a_h$ relative to the left-invariant vector field associated to $x$.
\begin{lemma}
	\label{lemma:curvatureonpsi}
	Let $\tilde G$ be a solvmanifold with a standard decomposition $\tilde \g=\g\rtimes\lie a$ of pseudo-Iwasawa type, and write the metric as
	\[\tilde g= g+\sum_\alpha \epsilon_\alpha e^\alpha\otimes e^\alpha.\]
	If $\psi=\sum_h a_h u_h$ is a Killing spinor on $\tilde G$ with  Killing constant $\lambda$ and $v\in \g$, then
	\begin{equation}
		\label{eqn:amended2lambdasquared}
		2\lambda^2 v\cdot e_\alpha\cdot \psi =\lambda \phi_\alpha v\cdot \psi -\sum_h (\Lie_{\phi_\alpha v}a_h) u_h.
	\end{equation}
\end{lemma}
\begin{proof}
	Since $\psi$ is Killing,
	\[
		\begin{split}
			\tilde R(v,e_\alpha)\psi & =\widetilde\nabla_{v}\widetilde\nabla_{e_\alpha} \psi-\widetilde\nabla_{e_\alpha}\widetilde\nabla_{v}\psi-\widetilde\nabla_{[v,e_\alpha]}\psi \\
			                         & =\lambda\bigl(\widetilde\nabla_{v}(e_\alpha\cdot\psi)-\widetilde\nabla_{e_\alpha}(v\cdot\psi)-\phi_\alpha v\cdot\psi\bigr)                    \\
			                         & =\lambda\bigl(\phi_\alpha  v \cdot\psi+\lambda e_\alpha \cdot v\cdot\psi-\lambda v\cdot e_\alpha\cdot \psi- \phi_\alpha v\cdot \psi\bigr)     \\
			                         & =-2\lambda^2 v\cdot e_\alpha\cdot \psi.
		\end{split}
	\]
	 Moreover, by Lemma~\ref{lemma:equazione8}, for any spinor $\psi$ and $v$ in $\g$,
	\[
		\begin{split}
			\tilde R(v,e_\alpha)\psi & = \frac14\sum_j \epsilon_j e_j\cdot \tilde R(v,e_\alpha)e_j\cdot\psi
			+\frac14\sum_\beta \epsilon_\beta e_\beta\cdot \tilde R(v,e_\alpha)e_\beta\cdot\psi                            \\
			                         & =-\frac14\sum_j \epsilon_j e_j\cdot (\widetilde\nabla_{\phi_\alpha v} e_j)\cdot\psi
			-\frac14\sum_\beta \epsilon_\beta e_\beta\cdot  (\phi_\beta\phi_\alpha v)\cdot\psi.
		\end{split}\]
	Comparing the two expressions, we obtain
	\begin{equation}
		\label{eqn:curvatureonpsi:intermediate}
		2\lambda^2 v\cdot e_\alpha\cdot \psi = \frac14\sum_j \epsilon_j e_j\cdot (\widetilde\nabla_{\phi_\alpha v} e_j)\cdot\psi
		+\frac14\sum_{\beta}  \epsilon_\beta e_\beta\cdot(\phi_\beta\phi_\alpha v)\cdot \psi.
	\end{equation}
	By~\eqref{eqn:ConnectionOnSpinorNabla}, for $w\in\g$  we have
	\[\begin{split}
			\widetilde\nabla_{w}\psi & =\sum_h (\Lie_{w}a_h) u_h + \frac14\sum_{j} \epsilon_j e_j \cdot (\widetilde\nabla_w e_j)\cdot \psi
			+ \frac14\sum_{\beta} \epsilon_\beta e_\beta \cdot (\widetilde\nabla_w e_\beta)\cdot \psi                                      \\
			                         & =\sum_h (\Lie_{w}a_h) u_h + \frac14\sum_{j} \epsilon_j e_j \cdot (\widetilde\nabla_w e_j)\cdot \psi
			+ \frac14\sum_{\beta} \epsilon_\beta e_\beta \cdot (\phi_\beta w)\cdot \psi,
		\end{split}\]
	so that
	\[ \frac14\sum_{j} \epsilon_j e_j \cdot (\widetilde\nabla_w e_j)\cdot \psi = \lambda w\cdot \psi -\sum_h (\Lie_{w}a_h) u_h
		- \frac14\sum_{\beta} \epsilon_\beta e_\beta \cdot (\phi_\beta w)\cdot \psi.\]
	Setting $w=\phi_\alpha v$ and substituting into~\eqref{eqn:curvatureonpsi:intermediate}, we obtain
	\[2\lambda^2 v\cdot e_\alpha\cdot \psi
		=\lambda \phi_\alpha v\cdot \psi -\sum_h (\Lie_{\phi_\alpha v}a_h) u_h.\qedhere\]
\end{proof}

We will need the following algebraic property of Clifford multiplication.
\begin{lemma}
	\label{lemma:end_dot_spin}
	Let $F\colon \R^{p,q}\to\R^{p,q}$ be a self-adjoint linear map, and suppose that for some nonzero $\psi$ in $\Sigma$,
	\[F(v)\cdot \psi=v\cdot\psi, \quad v\in \R^{p,q}.\]
	Then $F$ has $1$ as its only eigenvalue, i.e.\ $F-\id$ is nilpotent.
\end{lemma}
\begin{proof}
	Let $f=F-\id$, which is symmetric. By hypothesis, $f(v)\cdot\psi=0$ for all $v$, and the Clifford identity implies that $\im f$ is totally isotropic. Since $f$ is symmetric,
	\[\ker f=(\im f)^\perp\supseteq \im f;\] therefore, $f^2=0$.
\end{proof}
The key lemma in the proof of the main theorem is the following:
\begin{lemma}
	\label{lemma:abelian}
	Let $\tilde G$ be a solvmanifold with a pseudo-Iwasawa standard decomposition $\tilde \g=\g\rtimes\lie a$. If $\tilde G$ admits a Killing spinor  with $\lambda\neq0$, then $\g$ is abelian and each $\phi_\alpha^2\colon\g\to\g$ is a (Lie algebra) isomorphism which only has the eigenvalue $-4\epsilon_\alpha \lambda^2$.
	\end{lemma}
\begin{proof}
	Let $\psi=\sum_h a_h u_h$ be a Killing spinor with Killing constant $\lambda\neq0$.
	We can write~\eqref{eqn:amended2lambdasquared} as
	\begin{equation}
		\label{eqn:main:2lambdasquared}
		2\lambda^2 v\cdot e_\alpha\cdot \psi -\lambda \phi_\alpha v\cdot \psi +\sum_h (\Lie_{\phi_\alpha v}a_h) u_h=0.
	\end{equation}
	By Lemma~\ref{lemma:equazione8} and~\eqref{eqn:ConnectionOnSpinorNabla}, $\widetilde\nabla_{e_\beta}$ is zero on left-invariant vector fields and spinors. Thus, taking the covariant derivative with respect to $e_\beta$ of \eqref{eqn:main:2lambdasquared} yields
	\begin{equation}
		\label{eqn:main:cov}
		\begin{split}
			0 & =2\lambda^2 v\cdot e_\alpha\cdot \widetilde\nabla_{e_\beta}\psi -\lambda \phi_\alpha v\cdot \widetilde\nabla_{e_\beta}\psi +\sum_h (\Lie_{e_\beta}\Lie_{\phi_\alpha v}a_h) u_h                     \\
			  & =2\lambda^3 v\cdot e_\alpha\cdot e_\beta\cdot \psi -\lambda^2 \phi_\alpha v\cdot e_\beta\cdot\psi +\sum_h \bigl(\Lie_{[e_\beta,\phi_\alpha v]}a_h+\Lie_{\phi_\alpha v}\Lie_{e_\beta}a_h\bigr) u_h.
		\end{split}
	\end{equation}
	Again using~\eqref{eqn:ConnectionOnSpinorNabla}, we have
	\[\sum_h (\Lie_{e_\beta} a_h)u_h=\widetilde\nabla_{e_\beta}\psi=\lambda e_\beta\cdot\psi= \lambda \sum_h a_h e_\beta\cdot u_h.\]
	Since, for fixed $\beta$, $\{e_\beta\cdot u_h\}_h$ is another basis of left-invariant spinors,
	we deduce
	\[\sum_h (\Lie_{\phi_\alpha v} \Lie_{e_\beta} a_h)u_h = \lambda \sum_h  (\Lie_{\phi_\alpha v}a_h) e_\beta\cdot u_h.\]
	Thus,~\eqref{eqn:main:cov} and the definition of $\phi_\beta$ give
	\[
		0=2\lambda^3 v\cdot e_\alpha\cdot e_\beta\cdot \psi -\lambda^2 \phi_\alpha v\cdot e_\beta\cdot\psi -\sum_h (\Lie_{\phi_\beta\phi_\alpha v}a_h) u_h+\lambda e_\beta\cdot \sum_h (\Lie_{\phi_\alpha v}a_h)  u_h.
	\]
	Using~\eqref{eqn:main:2lambdasquared} to eliminate the Lie derivative terms, we obtain
	\[\begin{split}
0  &=2\lambda^3 v\cdot e_\alpha\cdot e_\beta\cdot \psi -\lambda^2 \phi_\alpha v\cdot e_\beta\cdot\psi +2\lambda^2\phi_\alpha v \cdot e_\beta\cdot\psi- \lambda\phi_\beta\phi_\alpha v\cdot \psi\\
&\qquad+\lambda e_\beta\cdot (-2\lambda^2v\cdot e_\alpha+ \lambda\phi_\alpha v)\cdot \psi                                          \\
	&=\bigl(2\lambda^3 v\cdot(e_\alpha\cdot e_\beta + e_\beta\cdot e_\alpha)  -\lambda\phi_\beta\phi_\alpha v\bigr)\cdot\psi \\
	&=(-4g(e_\alpha,e_\beta)\lambda^3 v  -\lambda\phi_\beta\phi_\alpha v)\cdot\psi.
	\end{split}\]
Recall that we are assuming  $\lambda\neq0$. Setting $\alpha=\beta$ we obtain
\[
\phi_\alpha^2(v)\cdot\psi=-4\epsilon_\alpha\lambda^2v\cdot\psi,\]
for all $v$ in $\g$. By Lemma~\ref{lemma:end_dot_spin}, $\phi_\alpha^2$ only has the eigenvalue $-4\epsilon_\alpha \lambda^2$; in particular, it is invertible. Moreover, $\phi_\alpha$ has at most two eigenvalues, either $\pm 2i\lambda$ or $\pm 2\lambda$, depending on the sign of $\epsilon_\alpha\lambda^2$.

	Now recall (see e.g. \cite[Chapter II]{Jacobson_1962}) that given a derivation and two generalized eigenspaces $\g_\lambda$, $\g_\mu$ (not necessarily distinct), either $\lambda+\mu$ is an eigenvalue and $[\g_\lambda,\g_\mu]\subseteq\g_{\lambda+\mu}$, or $[\g_\lambda,\g_\mu]=0$. In our case, the sum of two eigenvalues of $\phi_\alpha$ is not an eigenvalue. This implies that the Lie algebra $\g$ is abelian.
\end{proof}

\begin{lemma}
\label{lemma:onlyoneeigenvalue}
Let $\tilde G$ be a solvmanifold with a pseudo-Iwasawa standard decomposition $\tilde \g=\g\rtimes\lie a$. If $\tilde G$ admits a Killing spinor  with $\lambda\neq0$, then $\g$ is abelian, $\lie a=\Span{e_0}$, and $\phi_0=-\ad e_0$ has only one eigenvalue.
\end{lemma}
\begin{proof}
 	The fact that $\g$ is abelian is proved in Lemma~\ref{lemma:abelian}.  We will first show that $\dim \lie a$ must equal $1$.

	By Lemma~\ref{lemma:equazione8}, we have
	\[\widetilde\nabla_{e_\alpha} v = 0, \quad \widetilde\nabla_{e_\alpha} e_\beta=0, \quad \widetilde\nabla_w e_\alpha=\phi_\alpha(w), \quad \widetilde\nabla_w v = -\sum_\alpha g(\phi_\alpha w,v)\epsilon_\alpha e_\alpha,\]
	where as usual $v,w$ denote generic elements of $\g$.
	Let $\psi=\sum_h a_h u_h$ be a Killing spinor with Killing constant $\lambda$, so that~\eqref{eqn:ConnectionOnSpinorNabla} gives
	\[ \lambda e_\alpha\cdot\psi = \sum_h (\Lie_{e_\alpha} a_h)u_h  .\]
	For $v\in\g$, \eqref{eqn:ConnectionOnSpinorNabla} gives
	\begin{equation*}
	\begin{split}
			\lambda v\cdot\psi & = \sum_h (\Lie_v a_h)u_h - \frac14 \sum_j\epsilon_j e_j\cdot ( \sum_\alpha g(\phi_\alpha v,e_j)\epsilon_\alpha e_\alpha)\cdot \psi + \frac14  \sum_\alpha\epsilon_\alpha e_\alpha\cdot \phi_\alpha v\cdot\psi \\
			                   & = \sum_h (\Lie_v a_h)u_h - \frac14\sum_\alpha  (\sum_jg(\phi_\alpha v,e_j)\epsilon_j e_j)\cdot \epsilon_\alpha e_\alpha\cdot \psi + \frac14  \sum_\alpha\epsilon_\alpha e_\alpha\cdot \phi_\alpha v\cdot\psi  \\
			                   & = \sum_h (\Lie_v a_h)u_h - \frac14 \sum_\alpha\phi_\alpha v \cdot \epsilon_\alpha e_\alpha\cdot \psi + \frac14\sum_\alpha \epsilon_\alpha e_\alpha\cdot \phi_\alpha v\cdot\psi                                \\
			                   & = \sum_h (\Lie_v a_h)u_h - \frac12\sum_\alpha  \phi_\alpha v \cdot \epsilon_\alpha e_\alpha\cdot \psi.
		\end{split}
		\end{equation*}
		We now fix an index $\beta$ and proceed to eliminate the Lie derivatives in the last equality. Due to Lemma~\ref{lemma:abelian},  $\phi_\beta$ is an isomorphism; replacing $\alpha$ with $\beta$ and $v$ with $\phi_{\beta}^{-1}v$ in~\eqref{eqn:amended2lambdasquared}, we can write
	\begin{equation*}
	 2\lambda^2 \phi_{\beta}^{-1}(v)\cdot e_\beta\cdot\psi = \lambda v\cdot\psi - \sum_h (\Lie_{v }a_h)u_h.
	\end{equation*}
Thus,
\begin{equation}
\label{eqn:webetapsi}
 2\lambda^2 \phi_{\beta}^{-1}(v)\cdot e_\beta\cdot\psi+\frac12\sum_\alpha  \phi_\alpha v \cdot \epsilon_\alpha e_\alpha\cdot \psi =0.
\end{equation}
Now observe that if $f\colon\g\to\g$ is symmetric, then in the Clifford algebra we have
\begin{equation}
		\label{eqn:fsharpcdotpsi}
\sum_i (e^i)^\sharp \cdot f(e_i)=-\Tr(f).
\end{equation}
Indeed, if $\sum_i (e^i)^\sharp \otimes f(e_i)=\sum_{j,h}a_{jh} e_j\otimes e_h$, then $a_{jh}=a_{hj}$, so that
	\[\sum_{j,h}a_{jh}e_j\cdot e_h =\frac12 \sum_{j,h}a_{jh}(e_j\cdot e_h+e_h\cdot e_j)= \sum_j -a_{jj}\epsilon_j=-\Tr f.\]
Using \eqref{eqn:webetapsi} and \eqref{eqn:fsharpcdotpsi}, we compute
\begin{equation}
\label{eqn:allthetraces}
 \begin{split}
0&=\sum_i (e^i)^\sharp  \cdot\bigl( 2\lambda^2\phi_\beta^{-1}e_i\cdot e_\beta
			 			+ \frac12\sum_{\alpha}  \phi_\alpha e_i \cdot \epsilon_\alpha e_\alpha\bigr)\cdot \psi\\
 			&=(-2\lambda^2\Tr \phi_{\beta}^{-1}  -\frac12\epsilon_\beta \Tr\phi_\beta) e_\beta \cdot\psi   - \frac12\sum_{\alpha\neq\beta}(\Tr \phi_\alpha)\epsilon_\alpha e_\alpha\cdot\psi.
\end{split}
\end{equation}

The proof that $\dim\lie a=1$ proceeds by contradiction. Assuming that $\dim\lie a>1$, we  show that each $\phi_\alpha$ is traceless, which will enable us to compute the scalar curvature and derive a contradiction.

The first term in~\eqref{eqn:allthetraces} is zero: indeed, Lemma~\ref{lemma:abelian} implies that
$\phi_\beta^2+4\epsilon_\beta\lambda^2\id$ is nilpotent; since it commutes with $\phi_\beta^{-1}$,
\[\phi_\beta^{-1}(\phi_\beta^2+4\epsilon_\beta\lambda^2\id)=\phi_\beta +
4\epsilon_\beta\lambda^2\phi_\beta^{-1}\]
is also nilpotent and traceless. Therefore,
	\[-\sum_{\alpha\neq\beta}\Tr(\phi_\alpha)\epsilon_\alpha e_\alpha\cdot\psi=0. \]
	Setting $x=\sum_{\alpha\neq\beta}\Tr(\phi_\alpha)\epsilon_\alpha e_\alpha$, we find $0=x\cdot x\cdot \psi = -g(x,x)\psi$, hence
	\begin{equation}
		\label{eqn:sumalphaneqbeta}
		\sum_{\alpha\neq\beta} (\Tr (\phi_\alpha))^2\epsilon_\alpha=0.
	\end{equation}
	Summing over all $\beta$, we see that $\sum_\alpha (\Tr (\phi_\alpha))^2\epsilon_\alpha$ is also zero, and subtracting~\eqref{eqn:sumalphaneqbeta} we obtain $(\Tr(\phi_\beta))^2\epsilon_\beta=0$, i.e. $\Tr(\phi_\beta)=0$. Since $\beta$ is arbitrary, this implies that all $\phi_\alpha$ are traceless.

	We are now in a position to compute the scalar curvature. In the expression for $\widetilde\ric(v,w)$ in Proposition~\ref{prop:ricandrictilde}, each of the three terms vanishes: the first because $\g$ is abelian, the second because $\phi_\alpha$ is symmetric, and the last because $\Tr(\phi_\alpha)=0$. Thus,
	\[\widetilde\ric (v,w)=0, \qquad \widetilde\ric(e_\alpha,e_\beta)=- \Tr (\phi_\alpha\circ \phi_\beta),\]
	so that the scalar curvature of $\tilde G$ satisfies
	\[\tilde s=-\sum_\alpha \epsilon_\alpha \Tr((\phi_\alpha)^2)=4\lambda^2nm,\]
	where $n$ is the dimension of $\g$, $m>1$ the dimension of $\lie a$, and the trace is determined by Lemma~\ref{lemma:abelian}. By~\eqref{eqn:scalarifKilling},
	\[\tilde s=4\lambda^2 (n+m)(n+m-1),\]
	and we obtain
	\[(n+m)(n+m-1)=nm,\]
	which is absurd, because it has no solution for positive integers $m,n$.

	Thus, $\lie a$ is one-dimensional, generated by a unit vector $e_0$. Arguing as above, we obtain
	\[\widetilde\ric (v,w)=-\epsilon_0 g(\phi_0(v),w)\Tr\phi_0, \qquad \widetilde\ric(e_0,e_0)=- \Tr (\phi_0^2).\]
	Thus,
	\[4\lambda^2(n+1)n=\tilde s=-\epsilon_0(\Tr \phi_0)^2- \epsilon_0 \Tr (\phi_0^2) = -\epsilon_0(\Tr\phi_0)^2+4\lambda^2n.\]
	Up to reversing the sign of $e_0$, we can assume $\Tr \phi_0>0$, giving
	\[\Tr\phi_0 =  an,
		\qquad a=2\abs{\lambda}.
	\]
	By Lemma~\ref{lemma:abelian}, $\phi_0^2$ has only one eigenvalue, namely $\pm a^2$; hence, the eigenvalues of $\phi_0$ can only be $\pm a$, $\pm ia$; since their sum is $na$, the only possibility is that they all  equal $a$.
\end{proof}

In Section~\ref{sec:construction} we constructed a family of solvmanifolds $\tilde G^{a,k}$ endowed with Killing spinors, each determined by extending an element of a subspace $\Sigma^\pm_f\subset\Sigma$. The underlying Lie algebra has the form $\R^{p,q}\rtimes_{\phi_0}\R$, where $\phi_0$ only has the eigenvalue $a$, consistently with Lemma~\ref{lemma:onlyoneeigenvalue}. We are now ready to prove our main theorem, stating that no other possibilities occur.
\begin{theorem}
	\label{thm:main}
	Let $\tilde G$ be a solvmanifold with a pseudo-Iwasawa standard decomposition $\tilde \g=\g\rtimes\lie a$ and suppose that $\tilde G$ admits a Killing spinor. Then either:
	\begin{itemize}
	 \item $\tilde G$ is isomorphically isometric to the hyperbolic half-space $\tilde G^{a,0}$; or
	 \item $\tilde G$ is isomorphically isometric to $(\tilde G^{a,k},\tilde g_{\delta,\epsilon,\tau})$ with $a\ne0$, $k>0$, and the map from $\Sigma^\pm_f$ to the space of Killing spinors with Killing constant $\pm i\tau a/2$ defined in Proposition~\ref{prop:counterexamplestomaintheorem} is an isomorphism.
	\end{itemize}
If $\tilde G$ is Einstein, only the first case occurs.
\end{theorem}
\begin{proof}
By Lemma~\ref{lemma:onlyoneeigenvalue}, $\lie a=\Span{e_0}$, and $\phi_0$ acts with only one eigenvalue $a$; more precisely, Lemma~\ref{lemma:abelian} gives $a^2=-4\epsilon_0 \lambda^2$. It follows from Remark~\ref{remark:oneigenvalueimpliesGak} that $\tilde\g$ coincides with the Lie algebra of some  $\tilde G^{a,k}$. Since $\tilde G^{a,k}$ is simply connected, $\tilde G$ is a quotient by a central discrete subgroup; however, it is easy to check that $\tilde G^{a,k}$ has trivial center, hence $\tilde G$ is isometrically isomorphic to $\tilde G^{a,k}$ with the metric $\tilde g_{\delta,\epsilon,\tau}$.

If $\phi_0=a\id$, we obtain $\tilde G^{a,0}\cong H^\epsilon_a$, and every element of $\Sigma$ extends to a Killing spinor with Killing constant equal to $i\tau a/2$.
	Otherwise, write $\phi_0=a(\id +f)$, where $f$ is symmetric and nilpotent; necessarily we have $f^2=0$, so that
	\[\phi_0^{-1}=\frac1a(\id - f).\]
	 If $\psi$ is a Killing spinor, \eqref{eqn:webetapsi} gives
	 \[
	  \begin{split}
0&=\frac{2\lambda^2}{a} (\id - f)(v)\cdot e_0\cdot\psi + \frac {a}{2}(\id+f)(v)\cdot \epsilon_0 e_0\cdot\psi\\
&=-\frac{a\epsilon_0}{2} (\id - f)(v)\cdot e_0\cdot\psi + \frac{a}{2}(\id+f)(v)\cdot \epsilon_0 e_0\cdot\psi
=af(v)\cdot \epsilon_0 e_0\cdot\psi.\end{split}
	 \]
This implies $\im f\cdot\psi=0$. Thus, any Killing spinor $\psi$  can be written as $\psi=\sum a_h u_h$ where $\{u_h\}$ is a basis of $\Sigma_f$. Then
\eqref{eqn:ConnectionOnSpinorNabla} gives
\[
	 \lambda v\cdot \psi = \sum \Lie_v a_h u_h - \frac12\phi_0 v\cdot \epsilon_0 e_0\cdot\psi
\]
i.e.
\[	  v\cdot (\lambda+\frac a2\epsilon_0 e_0)\psi = \sum \Lie_v a_h u_h.\]
The right hand side is in $\Sigma_f$; for any fixed $w$ in $\R^{p,q}$ we obtain
\[f(w)\cdot v\cdot (\lambda+\frac a2\epsilon_0 e_0)\cdot\psi =0.\]
As $v\cdot f(w)\cdot (\lambda+\frac a2\epsilon_0 e_0)\cdot\psi$ is zero, this implies
\[2g(v,f(w))(\lambda+\frac a2\epsilon_0 e_0)\cdot\psi =0\]
for all $v,w$; since $f$ is nonzero, $\psi=-\frac {\epsilon_0 a}{2\lambda}e_0\cdot\psi=\pm i \tau  e_0\cdot\psi$, so the evaluation of $\psi$ at the identity lies in $\Sigma^\pm_f$.

Since Killing spinors with Killing constant $\lambda$ are parallel relative to the modified connection $\Xi_X\psi=\widetilde\nabla_X\psi- \lambda X\cdot\psi$, they are determined by their value at a point. Thus, every Killing spinor is obtained with the construction of Proposition~\ref{prop:counterexamplestomaintheorem}.
\end{proof}

As an application in the Riemannian case, we obtain an alternative proof of Corollary~\ref{cor:conbaum}. We repeat the statement for convenience.
\begin{corollary*}
	\label{cor:baumless}
	The only homogeneous Riemannian manifold carrying a Killing spinor with $\lambda$  imaginary is hyperbolic space.
\end{corollary*}
\begin{proof}
	As a consequence of~\eqref{eqn:RicciPsi}, the metric is Einstein of scalar curvature $s=4n(n-1)\lambda^2<0$. By the results of Heber, Lauret, B\"ohm and Lafuente that we recalled in Section~\ref{sec:iwasawa} (Remark~\ref{remark:RiemannianSolv}), we have a solvmanifold with a standard decomposition of pseudo-Iwasawa type. Thus, Theorem~\ref{thm:main} applies.
\end{proof}
In retrospect, we can conclude that the decomposition of a homogeneous Riemannian manifold carrying a Killing spinor with $\lambda$ nonzero as a warped product $F\times\R$ coincides with the decomposition as a semidirect product $G\rtimes\R$. If the two decompositions are known to coincide, then $F=G$ is a Ricci-flat Lie group, which immediately shows that it isometric to $\R^n$. The proof of Section~\ref{sec:homogeneous} amounts to showing that $F$ is a Lie group with a left-invariant metric; dually, the proof given above implies that the Killing spinor restricts to a parallel spinor on $G$.
\printbibliography

\medskip
\small\noindent D. Conti: Dipartimento di Matematica, Università di Pisa, largo Bruno Pontecorvo 6, 56127 Pisa, Italy.\\
\texttt{diego.conti@unipi.it}\\
\small\noindent F. A. Rossi: Dipartimento di Matematica e Informatica, Universit\`a degli studi di Perugia, via Vanvitelli 1, 06123 Perugia, Italy.\\
\texttt{federicoalberto.rossi@unipg.it}\\
\small\noindent R.~Segnan Dalmasso: Dipartimento di Matematica e Applicazioni, Universit\`a di Milano Bicocca, via Cozzi 55, 20125 Milano, Italy.\\
\texttt{romeo.segnandalmasso@unimib.it}

\end{document}